\documentclass[12pt,reqno]{amsart}

\usepackage{amsmath}
\usepackage{dsfont}
\usepackage{graphicx}
\usepackage[all]{xy}
\usepackage{amsthm}
\usepackage{amsfonts}
\usepackage{newlfont}
\usepackage{amscd}
\usepackage{amssymb}
\usepackage{latexsym}
\usepackage{eufrak}
\usepackage{euscript}
\usepackage{mathrsfs}
\usepackage{color}

\newtheorem{thm}{Theorem}[section]
\newtheorem*{thmA}{Main Theorem A}
\newtheorem*{thmB}{Main Theorem B}

\newtheorem{cor}[thm]{Corollary}
\newtheorem{lem}[thm]{Lemma}
\newtheorem{prop}[thm]{Proposition}
\newtheorem{ques}[thm]{Question}
\theoremstyle{definition}
\newtheorem{defn}[thm]{Definition}

\newtheorem{example}[thm]{Example}

\theoremstyle{remark}
\newtheorem{rem}[thm]{Remark}
\numberwithin{equation}{section}
\def\date{\ifcase\month\or January\or February \or March\or April\or May %
\or June\or July\or August\or September\or October\or November           %
\or December\fi\space\number\day, \number\year}
\catcode`\@=11             
\newfam\msbfam             
\newfam\euffam             
\font\tenmsb=msbm10        
\font\sevenmsb=msbm7       
\font\teneuf=eufm10        
\font\seveneuf=eufm7       
\textfont\msbfam=\tenmsb   
\scriptfont\msbfam=\sevenmsb%
\textfont\euffam=\teneuf   
\scriptfont\euffam=
\seveneuf                  

\def\F{\mathbb{F}}          
\def\N{\mathbb{N}}          
\def\Q{\mathbb{Q}}          
\def\T{\mathbb{T}}          
\def\R{\mathbb{R}}          
\def\Z{\mathbb{Z}}          
\def\dsr{{\mathbb R}}       
\def\dsz{{\mathbb Z}}       
\catcode`\@=12             
\def\card{\mathop{%
\mathrm{card}}\nolimits}    
\def\ssk{\smallskip}       
\def\msk{\medskip}         
\def\bsk{\bigskip}         
scaled\magstep1            
\def\.{{\cdot}}
\def\c #1{{\mathcal{S\hskip-1pt U\hskip-1pt B}}\left(#1\right)}
\def\cbig #1{{\mathcal{S\hskip-1pt U\hskip-1pt B}}\big(#1\big)}
\def\cab #1{{\mathcal{S\hskip-1pt U\hskip-1pt B}_{\mathrm{ab}}}\left(#1\right)}
\newcommand{\gen}[1]{\overline{\langle #1\rangle}}
\def\implies{\hbox{$\Rightarrow$}}

\def\hat{\widehat}
\def\til{\widetilde}
\def\defi{\buildrel\mathrm{def}\over=}

\newcommand{\set}[2]{\left\{#1\, \mid \, #2\right\}}

\def\cg{{\mathcal{S\hskip-.5pt U\hskip-.9pt B}}}
\def\cf{{\mathcal{F}_1}}
\def\cR{{\mathcal{R}_1}}
\def\comp{\mathrm{comp}}

\def\graph{\mathrm{graph}}
\def\id{\mathrm{id}}
\def\bR{{\mathrm b}\R}
\def\bZ{{\mathrm b}\Z}
\def\bR{{\mathrm b}\R}

\title[Approximable locally  compact  groups]{Locally  compact  groups
approximable by subgroups isomorphic to $\Z$ or $\R$}

\author[H. Hamrouni, K. H. Hofmann]{Hatem  Hamrouni and Karl H. Hofmann}

\address{Faculty of Sciences at Sfax.
Department of Mathematics.
Sfax University.
B.P. 1171. 3000 Sfax, Tunisia}
\email{hatemhhamrouni@gmail.com}

\address{Fachbereich Mathematik,
Technische Universit\"at Darmstadt,
Schlossgartenstrasse 7,
Darmstadt 64289, Germany
and
Department of Mathematics,
Tulane University,
New Orleans, LA 70118, USA}
\email{hofmann@mathematik.tu-darmstadt.de}


\newtheoremstyle{plain}
{10pt} 
{10pt} 
{\itshape} 
{} 
{\bfseries} 
{.} 
{.5em} 
{}

\begin{document}
\date{
      \vskip 20pt}

\begin{abstract} Let $G$ be a  locally compact
topological group, $G_0$ the connected component of
its identity element, and $\comp(G)$ the union of all compact
subgroups. A topological group  will be called
inductively monothetic if any subgroup generated (as
a topological group) by finitely many elements is generated
(as a topological group) by a single element.
The space $\cg(G)$  of all closed subgroups of $G$ carries
a compact Hausdorff topology called the Chabauty topology.
Let $\cf(G)$, respectively, $\cR(G)$, denote the subspace of
all discrete subgroups isomorphic to $\Z$, respectively,
all subgroups isomorphic to $\R$. 
It is shown that a necessary and sufficient
condition for $G\in\overline{\cf(G)}$ to hold is that $G$ is
abelian, and either that
$G\cong \R\times \comp(G)$ and $G/G_0$ is inductively
monothetic, or else that $G$ is discrete and isomorphic to a subgroup
of $\Q$. It is further shown that a necessary and sufficient
condition for $G\in\overline{\cR(G)}$ to hold is that
$G\cong\R\times C$ for a compact connected abelian group $C$.
\end{abstract}

\maketitle

\vskip-15pt

\centerline{Authors' addresses:}
\smallskip
\hbox{\vtop{\hsize=.45\hsize
Hatem Hamrouni\\ 
Faculty of Sciences at Sfax\\
Department of Mathematics\\
Sfax University\\
B.P. 1171. 3000 Sfax, Tunisia\\
{\tt hatemhhamrouni@gmail.com}}
\hfill\vtop{\hsize=.45\hsize
Karl H. Hofmann\\
Fachbereich Mathematik\\
Technische Universit\"at Darmstadt\\
Schlossgartenstrasse 7\\
Darmstadt 64289, Germany\\
{\tt hofmann@mathematik.tu-darmstadt.de}}}
\smallskip
\centerline{Running title:} 
\centerline{Locally Compact Groups Approximable by $\Z$ and $\R$}

\bigskip

\section{Preface}
\label{s:preface}

The simplest group arising directly from the
activity of counting is the group $\Z$ of integers.
On the other hand, one of the more sophisticated
concepts of group theory is that of a locally compact
topological group; it evolved widely and deeply
since David Hilbert in 1900 posed
the question whether a locally euclidean
topological group might  be parametrized differentiably
so that the group operations become differentiable.
Bringing $\Z$ and locally compact groups together
in a topologically systematic fashion is made possible
by a compact Hausdorff space
$\cg(G)$ attached to a locally compact group $G$
in a very natural fashion, namely, as the set of closed subgroups
endowed with  a suitably defined topology. Since $G$ itself
is a prominent element of $\cg(G)$, we pose and answer
completely the question under which circumstances
$G$ can be approximated in $\cg(G)$ by those subgroups
of $G$ which are isomorphic to $\Z$. A topological parameter
attached to a topological space $X$ is the smallest cardinal
of a basis for the collection of all open subsets;
this parameter is called the weight $w(X)$ of $X$.
If the weight of $X$ is not bigger than the first infinite
cardinal, one says that $X$ satisfies the Second Axiom of
Countability.
Our findings about the approximability of $G$ by
subgroups isomorphic to $\Z$ will show a fact that one
would not expect at first glance: Such an approximabilty does
not impose any bound whatsoever on the weight $w(G)$ of $G$.

\msk

If one ascends from the group $\Z$ of counting numbers to
the group $\R$ of real numbers which permits us to measure lengths
and distances, then the completely analogous question
suggests itself, asking indeed which locally compact groups $G$ can be
approximated in $\cg(G)$ by subgroups isomorphic to $\R$.
Again we answer that question completely and find, that the
answer to the second question yields a simpler structure than
the answer to the first question.

\msk

The answers are described in the Abstract which precedes our text.
  It is relatively immediate that
our discussion will usher us  into the domain of abelian locally compact
groups. But the final  proofs do  lead us more deeply
into the structure of these  groups than one might anticipate
and are, therefore, also longer than expected.
With these remarks, we now turn to the details.

\section{Preliminaries}
\label{s:prelim}

\subsection{Basic concepts and definitions} \label{ss:basics}
 Let  $G$  be  a  locally compact  group and $\comp(G)$ the union
of its compact subgroups.
By  Weil's  Lemma (\cite[Proposition 7.43]{hofmorr}),
an  element $g\in G$ is either contained in $\comp(G)$
or else the group $\langle g\rangle$ is isomorphic
as a topological group to $\Z$. Subgroups of this kind we shall
call \textit{integral}.
A subgroup $E$ of $G$ is called \textit{real} if it is isomorphic
 to $\R$ as a topological group.

\msk
 We  denote by $\c{G}$  the  space  of closed  subgroups  of  $G$
equipped with  the  \textit{Chabauty  topology};
this  is  a  compact  space. In this space,
each  closed subgroup~$H$ of~$G$ has a neighborhood base consisting of sets
\begin{equation}
\label{eq:Chabauty_base}
\mathcal{U}(H; K, W){\defi}\left\{L{\in}\cg(G) \mid L{\cap}K{\subseteq}WH
       \hbox{ and }H{\cap}K{\subseteq}WL\right\},
\end{equation}
where~$K$ ranges through the set $\mathcal K$ of all compact subsets
of~$G$ and~$W$ through the set $\mathcal U(e)$ of  all neighborhoods
of the identity.
In particular $G\in\cg(G)$ and the singleton subgroup $E=\{1\}$ have  bases for their
respective  neighborhoods of the form
\begin{eqnarray}
\label{eq:G_base}
\mathcal{U}(G; K, W) &=& \left\{L{\in}\cg(G)\mid  K{\subseteq}WL\right\},\label{eq:G_base}\\
\mathcal{U}(E; K, W) &=&\left\{L{\in}\cg(G)\mid L\cap K\subseteq W\right\},\label{eq:E_base}
\end{eqnarray}
$K\in{\mathcal K}$ and $W\in{\mathcal U}(e)$.

\begin{rem} \label{countability} Returning for a moment
to Equation (\ref{eq:Chabauty_base})
we assume  that $G$ is $\sigma$-compact and satisfies the First Axiom
of Countability. Then $G$ contains a sequence $(K_m)_{m\in\N}$
of compact subsets whose interiors for an scending sequence of open sets
covering $G$, and there is a sequence $(W_n)_{n\in\N}$ of open
identity neighborhoods forming a basis of the filter of identity neighborhoods.
Then each element $H\in\cg(G)$ has a countable basis
$\big(\mathcal{U}(H; K_m, W_n)\big)_{m,n\in \N}$ for its neighborhood filter
according to Equation (\ref{eq:Chabauty_base}).
Thus $\cg(G)$ satisfies the First Axiom of Countability.
\end{rem}

\msk

We denote by $\cf(G)$ the subspace of $\cg(G)$ containing all
$\langle g\rangle$ with $g\in G\setminus\comp(G)$, that is,
all subgroups isomorphic to the discrete group $\Z$ of all integers,
the {\it free group of rank} $1$.

\begin{example}[The additive group  $\dsr$] \label{e:R}
The mapping $\phi_{\dsr}\colon [0,\infty]\to\cg(\dsr)$ defined by
$$\phi_\R(r)=\begin{cases}%
\frac 1 r\.\Z  &\mbox{if } 0<r<\infty,\cr
    \{0\}      &\mbox{if } r=0,       \cr
    \R         &\mbox{if } r=\infty
\end{cases}$$
is a homeomorphism (see  Proposition 1.7 of \cite{Hattel1}).
Here $\comp(G)=\{0\}$, and $\cf(\R)=\{\langle r\rangle| 0<r\}$,
and if $(r_n)_{n\in \N}$ is  a  sequence
  of real  numbers converging  to  $0$, then
$(\langle r_n\rangle)_{n\in \N}$  converges  to  $\dsr$.
Therefore, $\R\in\overline{\cf(\R)}$
\end{example}

\begin{defn} \label{integral-subgroups}
A locally compact group $G$ is said to be
\textit{integrally approximable}
if $G\in\overline{\cf(G)}$.
\end{defn}

Here is an equivalent way of expressing that $G$ is
integrally approximable:

{\it There is a net
$(S_j)_{j\in J}$ of  subgroups isomorphic to $\Z$ in $G$
such that $G=\lim_{j\in J} S_j$ in $\cg(G)$.}

\begin{rem}\label{r:immediate}
If $G$ is integrally approximable, $\overline{\cf(G)}\ne\emptyset$,
and so $G\ne\comp(G)$. In particular, $G$ is not singleton.
\end{rem}

\medskip

Our  objective is to describe precisely which locally compact groups
are integrally approximable. The outcome is anticipated in the
abstract. It may be instructive for our intuition
to consider some examples right now.

\bigskip

\subsection{Some examples}\label{ss:examples}
We noticed in Example \ref{e:R} above that $\R$ is integrally approximable.
A bit more generally, we record

\begin{example} \label{p:vector-groups}
For $n\in\N$ and $G=\R^n$, the following statements are equivalent:
\begin{itemize}
\item[(1)] $G$ is integrally approximable.
\item[(2)] $n=1$.
\end{itemize}
\end{example}
\begin{proof} By Example \ref{e:R}, (2) implies (1).
For proving the reverse implication,  we define  the elements
$e_k=(\delta_{km})_{m=1,\dots,n}$, $k=1,\dots, n$ for the Kronecker
deltas
$$\delta_{km}=\begin{cases}%
1  &\mbox{if }\; k=m \\
0  &\mbox{otherwise.}\\
\end{cases}$$
Now we assume (1) and $n\ge2$ and propose to  derive a contradiction.

We let
$W\in{\mathcal U}((0,\dots,0))$ be the open ball of
radius $\frac 1 2$ with respect to the euclidean metric on
$\R^n$,  and let $K$ be the closed ball of radius 2 around $(0,\dots,0)$.
 By (1) there is an
 integral  subgroup $S$ in the neighborhood
$\mathcal U(G;K,W)$ of $G$ in $\cg(G)$.

In view of Equation (\ref{eq:G_base}) this means
$K\subset W+S$. Since $n\ge 2$ we know that
$e_1$ and $e_2$ are contained in $K$ and therefore in $W+S$, that is
there are elements $w_1, w_2\in W$ such that $s_1=e_1-w_1$
and $s_2=e_2-w_2$ are contained $S$. Now the euclidean
distance of $s_m$ from $e_m$ for $m=1,2$ is $<{\frac 1 2}$
in the euclidean plane $E_2\cong\R^2$ spanned by $e_1$ and $e_2$.
Therefore, the two elements $s_1$ and $s_2$ are linearly
independent. On the
other hand, being elements of the subgroup $S\cong \Z$,
they must be linearly dependent.
This contradiction proves that (1) implies (2).
\end{proof}

\begin{example} \label{e:rational} The group $\Q$ (with the discrete topology) is integrally approximable.
\end{example}

\begin{proof} For each natural number $n$ set $H_n=\frac{1}{n!}\Z$.
Since   $(H_n)$  is  an  increasing  sequence  and
\begin{equation*}
\bigcup_{n\in \N} H_n = \Q,
\end{equation*}
we  also have $\lim_{n\in \N} H_n = \Q$ as we may conclude
directly using (\ref{eq:G_base}) or by invoking
Proposition 2.10 of \cite{Hamr-Sad-JOLT}. This proves the claim.
\end{proof}

\begin{example} \label{e:R-times-four} The group $G=\R\times\Z(p)^n$
for any prime $p$  and any natural number $n\ge2$
 is not approximable by  integral subgroups. The smallest example
in this class is $\R\times\Z(2)^2$.
\end{example}

\begin{proof} We note that $\Z(p)^n$ is a vector space $V$ of
dimension $n$ over the field $\F=\mathrm{GF}(p)$.
By way of contradiction suppose that
$G$ is integrally approximable. Let $W=[-1,1]\times\{0\}$
and let $K=\{0\}\times B$ for  a basis $B$ of $V$.
Since $G$ is integrally approximable, there is a $Z\in \cf(G)$
such that $Z\in \mathcal U(G; K, W)$, that is,
$K\subseteq W+Z$. There is an $r\in\R$ and a $v\in V$
such that $Z=\Z\.(r,v)$. Thus
\begin{multline*}
\{0\}\times B\subseteq ([-1,1]\times\{0\})+\Z\.(r,v)\\
=\bigcup_{n\in\Z} \{nr+[-1,1]\}\times \{n\.v\}
\subseteq \R\times \F\.v,
\end{multline*}
 and thus $B\in \F\.v$.
This implies $\dim_\F V\le 1$ in contradiction to
the assumption $n\ge2$.
\end{proof}

For any prime $p$ we recall the basic groups
$\Z(p^n)$, $n=1,\dots,\infty$, and $\Z_p\subset\Q_p$,
where $\Z(p^\infty)$ is the divisible Pr\"ufer group and
$\Z_p$ its character group, the group of $p$-adic integers
(see. e.g.\ \cite{hofmorr}, Example 1.38(i), p.~27) and where
$\Q_p$ denotes the group of $p$-adic rationals
(see loc.~cit. Exercise E1.16, p.~27). Then Example \ref{e:R-times-four}
raises at once the following question:

\begin{ques} \label{q:question} Are the  locally compact
groups $G=\R\times C$ for $C=\Z(p^n)$, $C=\Z(p^\infty)$, $C=\Z_p$, or
$C=\Q_p$ integrally approximable?
\end{ques}

In the light of the negative Example \ref{e:R-times-four}, this
may not appear so simple a matter to answer. Our results will show that
they all are integrally approximable.

\medskip

While the group $\Z$ is integrally approximable trivially from
the definition, we have, in the context of the group $\Z$,
the following lemma, whose proof we
can handle as an exercise directly from the definitions and which serves
as a further example of the particular role played by $\Z$ in
the context of integrally approximable groups.

\begin{lem}  \label{l:G=Z} Let $A$ be a locally compact
abelian group and assume that $A\times \Z$
is integrably approximable. Then $|A|=1$.
\end{lem}
\begin{proof} By way of contradiction assume that there is an
$a\ne 0$ in $A$. Then there is a zero-neighborhood $W=-W$ in $A$
such that $a\notin W+W$. Set $G=A\times\Z$ and
 $K=\{(a,1),(0,1)\}\subseteq G$.
Since $G$ is integrally approximable there is
a $Z\cong\Z$ in $\cf(G)$ such that
$$Z\in{\mathcal U}(G; K,W\times\{0\}).$$
Since $Z$ is an integral subgroup, there are
elements $b\in A$ and $0<n\in\Z$ such that
$Z=\Z\.(b,n)$. So $(a,1)$ and $(0,1)$ are contained in
$(W\times\{0\})+\Z\.(b,n)$. Therefore
there are elements
 $w_a, w_0\in W$
and integers $m_a, m_0\in\Z$
such that
\begin{enumerate}
\item[$(\mathrm{i})$] $w_0+m_0\.b=0$,
\item[$(\mathrm{ii})$] $w_a+m_a\.b=a$, and
\item[$(\mathrm{iii})$] $m_0n=1$.
\item[$(\mathrm{iv})$]  $m_an=1$.
\end{enumerate}

Equations (iii) and (iv), holding in $\Z$, imply
$m_0=m_a=n=1$.
Thus equation (i) implies $b\in -W=W$. Then from (ii) it follows that
$a\in W+W$. This is a contradiction, which proves our claim for the
example.
\end{proof}

\bigskip

\section{The class of approximable groups}

\label{s:class}

As a first step towards the main results it will be helpful
to observe the available closure properties of the class of
integrally approximable groups.

\subsection{Preservation properties} \label{ss:preserv}

\begin{prop}\label{p:class(appro-integral)}
The class of integrally approximable locally compact groups
is closed under the following operations:
\begin{enumerate}
\item[(OS)] Passing to open nonsingleton subgroups,
\item[(QG)] passing to quotients modulo compact subgroups,
\item[(QO)] passing to torsion-free quotients modulo open subgroups,
\item[(DU)] forming directed unions of closed subgroups,
\item[(PL)] forming strict projective limits of quotients
            $G/N$ modulo compact subgroup $N$.
\end{enumerate}
\end{prop}

\begin{proof} (OS) Let $G$ be a locally compact group
$G$ that is approximable by integral subgroups and  let $U$ be
an open nonsingleton subgroup. Let $K$ be a compact subspace of
$U$  and $W$ an identity neighborhood contained in $U$; we
may take $K\ne\{1\}$ and $W$ small enough so that $K\not\subseteq W$.
 We must
find a subgroup $Z\cong\Z$ inside $U$ such that
$Z\in \mathcal U(U;K,W)$, that is, $K\subseteq WZ$.

However, $G$ is integrally approximable, and so there is a
subgroup $E\cong\Z$ of $G$ such that $E\in \mathcal U(G;K,W)$,
that is, $K\subseteq WE$. Then $K=K\cap U\subseteq WE\cap U$.
By the modular law, $W(E\cap U)=WE\cap U$ and  so
$K\subseteq W(E\cap U)$. The condition $K\not\subseteq W$
rules out the possibility that $E\cap U=\{1\}$. Since
$E\cong\Z$ we know that $E\cap U\cong\Z$ and so we may take
$Z=E\cap U$ and thus obtain $K\subseteq WZ$ which is what
we have to prove.
\msk

(QG)  Let $N$ be a compact  subgroup of  a  locally  compact
group  $G$  approximable by integral subgroups  and
let $\pi\colon G\to H$, $H=G/N$, be the quotient morphism. Then
the  continuity of  the
map $A\mapsto \pi(A):\c{G}\to \c{H}$ (see Corollary 2.4  of \cite{Hamr-Kad-JOLT}) implies that $H$
is approximable by integral subgroups.

\msk

(QO) Let $U$ be an open  subgroup of  a  locally  compact
group  $G$  approximable by integral subgroups
so that $H\defi G/U$
is torsion-free,  and
let $\pi\colon G\to H$  be the quotient morphism.
Since $U$ is open, $H$ is discrete, and the singleton
set containing the identity $\til e =U$ in $H=G/U$ is an
identity neighborhood. So for a given compact, hence finite,
subset $\til K$ of  $H$ we have to find a $\til Z\in \cf(H)$
with $\til Z \subseteq \mathcal U(H;\til K,\{\til e\})$,
that is $\til K\subseteq \til Z$ according to
Equation (\ref{eq:G_base}). We may and will assume that there
is at least one $\til k\in \til K$ such that $\til k\ne\til e$.

Now by the local compactness of $G$ we find  a compact
set  $K$  of  $G$  such  that $\pi(K)= \til K$.

Since $U$ is an identity neighborhood  in $G$
and since $G$ is integrally approximable, there is a
subgroup $Z\in\cf(H)$ contained in
$\mathcal{U}(G; K, U)$, that is $K\subseteq UZ$. Applying $\pi$,
we get $\til K\subseteq \pi(Z)$. In particular,
$\til e\ne\til k\in \pi(Z)$. Since $H$ is torsion-free
and $Z\cong\Z$ we conclude that $\pi(Z)\cong\Z$, and so
we can set $\til Z=\pi(Z)$, getting $\til K\subseteq \til Z$,
which we had to show.

\msk

(DU)  Let  $(G_i)_{i\in I}$  be  a  directed family of  closed
subgroups of  a locally compact group $G$ such that
$G=\overline{\bigcup_{i\in I} G_i}$.
Then from Proposition 2.10 of \cite{Hamr-Sad-JOLT} we know
\begin{equation} \label{eq:no-one}
G=\lim_i G_i\mbox{ in $\c{G}$}.
\end{equation}
Now assume
\begin{equation} \label{eq:no-two}
(\forall i\in I)\ ~G_i \mbox{ is approximable  by integral
subgroups.}
\end{equation}
 Let $\mathcal U$ be an open
neighborhood of $G$ in $\c{G}$. Then by (\ref{eq:no-one}) there is some $j\in I$ with
$G_j\in {\mathcal U}$, and so ${\mathcal U}$ is also an open neighborhood
of $G_j$ in $\c{G}$. Then by (\ref{eq:no-two}) there is a closed subgroup $Z\cong\Z$
in $G$ with $Z\in{\mathcal U}$. This proves that $G$ is approximable
by integral subgroups.

\msk

(PL) Let the locally compact group $G$ be a strict projective
limit $G=\lim_{N\in{\mathcal N}}G/N$ of integrally approximable
quotient groups modulo compact
normal subgroups $N$. Then $G$ has arbitrarily small open identity
neighborhoods $W$ for which there is an $N\in{\mathcal N}$ such that
$W=NW$. Let $K\in\mathcal K$ be a compact subspace of $G$. If $W$ is
given, we note that $NK$ is still compact, and so we assume that
$NK=K$ as well. We aim to show that there is a subgroup $Z\cong\Z$
of $G$ such that $K\subseteq WZ$ which will show that
$Z\in{\mathcal U}(G;K,W)$ as in Equation (\ref{eq:G_base}), and  this will
complete the proof.

Now we assume that for all $M\in\mathcal N$ the group $G/M$
is approximable by integral subgroups. Then, in particular,
$G/N$ is approximable by integral subgroups. Therefore
we find a subgroup $Z_N\subseteq G/N$ such that $Z_N\cong\Z$ and
$Z_N\in {\mathcal U}(G/N; K/N, W/N)$ according to
(\ref{eq:G_base}). This means that
\begin{equation} \label{eq:in N}
K/N\subseteq (W/N)\.Z_N.
\end{equation}
Let $z\in G$ be such that $Z_N=\gen{zN}$. Set $Z=\gen z$, then
$Z\cong\Z$ and $Z_N=ZN/N$. Then (\ref{eq:in N}) is equivalent
to $K/N\subseteq (W/N)(ZN/N)=WZ/N$ which in turn is equivalent
to $K\subseteq WZ$ and this is what we had to show.
\end{proof}

\begin{rem}
In the proof of (QO) it is noteworthy that we did not invoke
an argument claiming the continuity of the function
$A\mapsto AU/U:\cg(G)\to \cg(H)$. Indeed
 let $\pi\colon\R\times \Z\to \Z$ be the projection. The  sequence
$(\langle(n, 1)\rangle)_{n\in\N}$
converges to  the  trivial  subgroup $\{(0,0)\}$ in $\cg(G)$, but
its image is the constant sequence with value $\langle 1\rangle=\Z$
and therefore   converges to  $\Z$. This shows that the induced
map $\cg(\pi)\colon\cg(G)\to\cg(H)$ need not be continuous in general.
\end{rem}

\section{The case of discrete groups}
\label{s:discr}
In order to demonstrate the workings of some of the operations
discussed in Proposition \ref{p:class(appro-integral)} we show

\begin{lem} \label{l:discreteness}
Assume that an integrally approximable locally compact
group $G$ has a compact identity component $G_0$. Then $G$ is discrete.
\end{lem}
\begin{proof} By way of contradiction suppose that $G$ is not
discrete.
Since $G_0$ is compact, there is a compact open subgroup $U$
(see \cite{Montgomery--Zippin}, Lemma 2.3.1 on p.~54). Since $G$ is not
discrete, $U\ne\{1\}$. Since
$G$ is integrally approximable, so is $U$ by
Proposition \ref{p:class(appro-integral)} (OS).
Then $U\ne\comp(U)$ by Remark \ref{r:immediate}, but $U$ being
compact we have $U=\comp(U)$ and this is a contradiction which
proves the lemma.
\end{proof}

\subsection{Monothetic and inductively monothetic groups}
\label{ss:mon}

Before we proceed we need to recall some facts around
monothetic groups. A topological group $G$ is {\it monothetic}
if there is an element $g\in G$ such that $G=\gen g$.

\begin{defn}[Inductively monothetic group]\label{d:ind-mon} A topological group  $G$  is called
{\em inductively monothetic} if (and only if)  every finite
subset $F\subseteq G$ there is an element $g\in G$
such that $\gen F =\gen g$.
\end{defn}

The circle group $\T=\R/\Z$ contains a unique element
$t=2^{-1}+\Z\in\T$ such that $2\.t=0$. The group $\T^2$
 is monothetic but not inductively monothetic,
since the subgroup $\langle\{(t,0), (0,t)\}\rangle$ is finitely
generated but not monothetic. The discrete additive  group $\Q$
is inductively monothetic but is not monothetic.

In \cite{HHR}, Theorem 4.12 characterizes inductively monothetic
locally compact groups. Before we cite this result we recall
that Braconnier (see \cite{Braconnier}) called a locally
compact group $G$ a {\it local product}
$\prod_{j\in J}^{\mathrm{loc}}(G_j,C_j)$ for a family $(G_j)_{j\in J}$
of locally compact groups $G_j$ if each of them has a compact open
subgroup $C_j$ such that an element $g=(g_j)_{j\in J}\in\prod_{j\in J}G_j$
is in $G$ iff there is a finite subset $F_g\subseteq J$ such that
$g_j\in C_j$ whenever $j\notin F_g$.   

\begin{prop}[Classification of inductively monothetic groups]\label{p:ind-mon} A locally compact group $G$
is  inductively monothetic if one of the following conditions is
satisfied:
\begin{itemize}
\item[(1)] $G$ is  a one-dimensional compact connected group,
\item[(2)] $G$ is discrete and is isomorphic to
            a subgroup of $\Q$.
\item[(3)] $G$ is isomorphic to a local product
           $\prod_{p\ \mathrm{prime}}^{\mathrm{loc}}(G_p,C_p)$
           where each of its characteristic $p$-primary components $G_p$
           is either $\cong \Z(p^n)$, $n=0,1,\dots,\infty$, or $\Z_p$,
           or $\Q_p$.
\end{itemize}
\end{prop}

It follows, in particular, that a \emph{totally disconnected compact
monothetic group is inductively monothetic}  so that the concept
of an inductively monothetic locally compact group is more general
than that of a locally compact monothetic group \emph{in the
totally disconnected domain}.
\ssk We remark that a locally compact group is called {\it periodic} if
it is totally disconnected and has no subgroups isomorphic to $\Z$.
Thus the class (3) of Proposition \ref{p:ind-mon} covers precisely the
periodic inductively monothetic groups.

Our present stage of information  allows us to clarify on an elementary level
the discrete side of our project:

\begin{thm} \label{th:discrete-case}  Let  $G$  be  a locally compact group
such that $G_0$ is a compact  group.
 Then the  following assertions  are  equivalent:
\begin{enumerate}
  \item[$(1)$] $G$ is integrally approximable.
  \item[$(2)$] $G$ is discrete and isomorphic to a nonsingleton
               subgroup of $\Q$.
\end{enumerate}
\end{thm}

\begin{proof} In Example \ref{e:rational}
we saw that $\Q$ is integrally approximable. Then from
Proposition \ref{p:class(appro-integral)} (OS) if follows
that every nonsingleton subgroup of $\Q$ is integrally
approximable. Thus   (2) \implies (1).

We have to show  $(1)\implies (2)$:\quad
Thus we assume (1). In particular, $G$ is nonsingleton.
Since the subgroup $G_0$ is compact, Lemma \ref{l:discreteness}
applies and shows that $G$ is discrete.
Then by Equation \ref{eq:G_base} in $\cg(G)$ the element
$G$ has a basis of neighborhoods
$\mathcal U(G; F, \{0\})=\{H\in\cg(G): F\subseteq H\}$ as
$H$ ranges through the finite subsets of $G$.
Since $G$ is integrally approximable, there exists
a $Z\in \cf(G)$ such that $Z\in\mathcal U(G;F,\{0\})$,
that is, $F\subseteq Z$. Then $\langle F\rangle$ is infinite
cyclic as a subgroup of a
group $\cong\Z$. Therefore $G$ is discrete, torsion-free,
 and inductively monothetic.
Then Proposition \ref{p:ind-mon} shows that $G$ is isomorphic to
a subgroup of $\Q$.
\end{proof}

\section{Necessary conditions}
\label{s:necessary}

For the remainder  of the effort to classify integrally approximable
groups we may therefore concentrate on nondiscrete groups, and
indeed on locally compact groups $G$ whose identity component
$G_0$ is noncompact.

\subsection{Background  on abelian
                locally compact groups}
\label{ss:commutativity}

We first point out why we have to focus on commutative
locally compact groups.
Indeed in \cite[Proposition 3.4]{Harp1} the
following fact was established:

\begin{prop}\label{space-abelian-closed}
Let  $G$  be  a  locally  compact  group.
Then the  space $\cab{G}$  of closed
abelian subgroups of  $G$  is  closed  in $\c{G}$.
\end{prop}

We have $\cf(G)\subseteq\cab{G}$ and thus
$\overline{\cf(G)}\subseteq \cab{G}$.
Accordingly, in view of Definition \ref{integral-subgroups}
we therefore have

\begin{cor} \label{focus-abelian}
An integrally approximable locally compact group is abelian.
\end{cor}

Thus we now focus on locally compact abelian groups and their
duality theory. As a consequence we shall henceforth write
the groups we discuss in additive notation.
An example is the following result of Cornulier's
(see \cite{Cornulier-LCA-Chabauty}, Theorem 1.1):

\begin{prop}[Pontryagin-Chabauty Duality] \label{Pon-Cha-duality}
Let  $G$  be  an  abelian  locally  compact  group. Then  the
annihilator map
\begin{equation*}
\ H\mapsto H^\bot : \c{G}\to \cbig{\hat G}
\end{equation*}
is  a homeomorphism.
\end{prop}

This will allow us to apply the so-called {\it annihilator mechanism}
as discussed e.g. in \cite{hofmorr}, 7.12 ff., pp.314 ff..

What will be relevant in our present context is a main
structure theorem for abelian groups (see \cite{hofmorr},
Theorem 7.57, pp. 345 ff.).

\begin{prop} \label{p:vector-splitting} Every
locally compact abelian group  $G$ is algebraic\-al\-ly and topologically of
the form $G=E\oplus H$ for a subgroup $E\cong \R^n$ and a locally compact
abelian subgroup $H$ which has the following properties
\begin{itemize}
\item[\rm(a)] $H$ contains  a compact  subgroup which is open in $H$.
\item[\rm(b)] $H$ contains $\comp(G)$.
\item[\rm(c)] $H_0=(\comp(G))_0=\comp(G_0)$ is the unique maximal
              compact connected subgroup of $G$.
\item[\rm(d)] The subgroup $G_1\defi G_0+\comp(G)$ is an open,
              hence closed, fully characteristic
              subgroup which is isomorphic to $\R^n\times \comp(G)$.
\item[\rm(e)] $G/G_1$ is a discrete torsion-free group and
              $G_1$ is the smallest open subgroup with this property.
\end{itemize}
 \end{prop}

\begin{gather}
\begin{aligned}
\xymatrix{
&&&G \ar@{-}[dr]  \ar@{-}[dl]\\
&&G_1 \ar@{-}[dr] \ar@{-}[dl]&&H\ar@{-}[dl]\\
&G_0  \ar@{-}[dr] \ar@{-}[dl] &&C\ar@{-}[dl]\\
E\ar@{-}[dr]&&C_0 \ar@{-}[dl]\\
&0
}\end{aligned}
\label{Fig1}
\end{gather}
\begin{center}
{$C=\comp(G),\quad G_1=G_0+C=E\oplus C.$}
\end{center}


\subsection{Necessity}
\label{ss:nec}

These results allow us to narrow our scope onto integrally
approximable groups  $G$ further and to derive necessary
conditions for $G$ to be integrally approximable.

\begin{prop} \label{p:vector-splitting-int-appr} Every nondiscrete
integrally approximable locally compact abelian group  $G$
is algebraically and topologically of
the form $G=E\oplus\comp(G)$ for a subgroup $E\cong \R$ and
the locally compact abelian subgroup $\comp(G)$.
 \end{prop}

\begin{proof} (i) Using the notation of
Proposition \ref{p:vector-splitting} (d) and (e) above we first claim
$G=G_1$. By (d) the subgroup $G_1=G_0+\comp(G)$ is open, and by
(e) the factor group $G/G_1$ is torsion free. Suppose our claim
is false. Then we find an element $g\in G\setminus G_1$. Then
$Z\defi \langle g\rangle$ is  cyclic and $(Z+G_1)/G_1
\cong Z/(Z\cap G_1)$ is  torsion-free by (e), and so $X\cap G_1=\{0\}$.
Then it follows from (d), since $G_1$ is open, that $Z$ is discrete.
Hence $U=G_1 +Z$ is a nonsingleton open subgroup $\cong G_1\times \Z$
by (e). Then by Lemma \ref{l:G=Z} we have $G_1=\{0\}$ which forces
$G$ to be discrete, contrary to our hypothesis.

(ii) Now by Proposition \ref{p:vector-splitting} again,
 we may identify $G$ with $\R^n\times H$ where $H_0$ is
compact and $H=\comp(H)$. Now $H$ contains a compact
open subgroup $U$ (cf.\ the proof of Lemma \ref{l:discreteness})
and $\R^n\times U$ is an open subgroup of $G$ which is nonsingleton
since $G$ is nondiscrete. Hence \break $\R^n\times U$ is integrally
approximable by Proposition \ref{p:class(appro-integral)}.
Then the projection $\R^n\times U\to\R^n$ is covered by part
(QG) of Proposition \ref{p:class(appro-integral)} and thus
we know that $\R^n$ is integrally approximable.
Then Example \ref{p:vector-groups} shows that $n=1$.
\end{proof}

\begin{rem} \label{r:iso} If the relation $G=G_0+\comp(G)$
holds for a locally compact abelian group $G$, then
 $G/G_0\cong \comp(G)/\comp(G)_0$. Moreover,
 $\comp(G)$ is the union of compact subgroups being open
in $\comp(G)$.
\end{rem}

From here on we concentrate on groups of the form
$\R\times H$ where $H$ is an abelian locally compact
group with $H=\comp(H)$.

\begin{defn} \label{d:periodic} We call a locally compact group
$H$ {\it periodic} if it is totally disconnected and satisfies
$H=\comp(H)$.
\end{defn}

Periodic abelian groups have known structure due to Braconnier
(see \cite{Braconnier}; cf.\ also \cite{HHR}).
Indeed, a periodic group $G$ is (isomorphic to) a local product
$$\prod_{p\ \mathrm{prime}}^{\mathrm{loc}}(G_p,C_p)$$
for the $p$-primary components (or $p$-Sylow subgroups)
$G_p$.

\begin{lem} \label{l:monothetic}
Let $G$ be an integrally approximable locally compact
group such that $\comp(G)$ is periodic and compact.
Then $\comp(G)$ is monothetic.
\end{lem}

\begin{proof} Write $H=\comp(G)$. Then we may identify
$H$ with the product $\prod_p H_p$ of its compact $p$-primary
components (see also \cite{hofmorr}, Proposition 8.8(ii)).
A compact group $H$ is monothetic iff there is a morphism
$\Z\to H$ with dense image iff (dually) there is an injective morphism
$\hat H\to \T$. As $H$ is totally disconnected, $\hat H$ is
a torsion group (see \cite{hofmorr}, Corollary 8.5, p.~377), and so the group
$\hat H$ is embeddable into $\T$ iff it is embeddable into
the torsion group $\Q/\Z=\bigoplus_p\Z(p^\infty)$ of $\T$ (see
\cite{hofmorr}, Corollary A1.43(ii) on p.~694) iff each $\hat H_p$
is embeddable into $\Z(p^\infty)$. Hence $H$ is monothetic iff
 each $H_p$ has $p$-rank $\le 1$. By way of contradiction suppose
that this is not the case and that there is a prime $p$ such that
the $p$-rank of $H_p$ is $\ge2$. So  the compact
group $H/p\.H$ has exponent $p$ and  is isomorphic to a power
$\Z(p)^I$ with $\card I\ge2$. Therefore we have a projection of
$H/p\.H\cong\Z(p)^I$ onto $\Z(p)^2$. This provides a surjective
morphism $H\to H/p\.H\cong \Z(p)^I\to \Z(p)^2$.
Thus Proposition \ref{p:vector-splitting-int-appr}
tells us that $G=\R\times H$ has a quotient group  $\R\times\Z(p)^2$ modulo
a compact kernel. By Proposition \ref{p:class(appro-integral)} (QG)
this quotient group is integrally approximable
 which we know to be impossible by
Example \ref{e:R-times-four}. This contradiction proves
the lemma.
\end{proof}

\begin{lem} \label{l:ind-mon} Let $G$ be a totally disconnected
locally compact abelian group  satisfying $G=\comp(G)$.
Assume that every compact open subgroup of $G$ is monothetic.
Then $G$ is inductively monothetic.
\end{lem}
\begin{proof}
Let $F$ be a finite subset of $G$; we must show that
$\gen F$ is monothetic.
Let $\mathcal S$ be the  $\subseteq$-directed set of
all compact open (and therefore monothetic) subgroups of $G$.
Since $G=\bigcup{\mathcal S}$ (see Remark \ref{r:iso})
for each $x\in F$ there is a $C_x\in{\mathcal S}$ such that
$x\in C_x$. Since ${\mathcal S}$ is directed and $F$ is finite
there is an $K\in{\mathcal S}$ such that
$\bigcup_{x\in F}C_x\subseteq K$. Then $F\subseteq K$ and so
$\gen F\subseteq K$. Since $G$ is totally disconnected, the same
is true for $K$. By hypothesis, $K$ is monothetic, and so
the comment following Proposition \ref{p:ind-mon}
shows that $K$ is inductively
monothetic, whence $\gen F$ is monothetic.
\end{proof}

\begin{cor} \label{c:ind-mon-2} Let $G$ be an integrally
approximable group such that $\comp(G)$ is periodic.
Then $G/G_0\cong \comp(G)$ is inductively monothetic.
\end{cor}

\begin{proof} Let $U$ be a compact-open subgroup of
$\comp(G)$. Then $G_0+U\cong \R\times U$ is an open subgroup of
$G=G_0+\comp(G)=\R\times\comp(G)$. Hence it is integrally
approximable by Proposition \ref{p:class(appro-integral)} (OG).
 So $U$ is monothetic by Lemma \ref{l:monothetic}. Now
Lemma \ref{l:ind-mon} shows that $\comp(G)$ is inductively
monothetic.
\end{proof}

Now we have a necessary condition on a locally compact group to
be integrally approximable:

\begin{thm} \label{th:necessary} Let $G$ be a nondiscrete integrally
approximable locally compact group. Then
\begin{itemize}
\item[\rm(a)] $G\cong \R\times \comp(G)$ and
\item[\rm(b)] $G/G_0\cong  \comp(G)/\comp(G)_0$ is inductively monothetic.
\end{itemize}
\end{thm}

\begin{gather}
\begin{aligned}
\xymatrix{
&&G \ar@{-}[dr] \ar@{-}[dl]&&\\
&G_0  \ar@{-}[dr] \ar@{-}[dl] &&C\ar@{-}[dl]\\
\R\ar@{-}[dr]&&C_0 \ar@{-}[dl]\\
&0
}\end{aligned}
\label{Fig2}
\end{gather}
\begin{center}
{$C=\comp(G),\quad G=G_0+C=\R\oplus C.$}
\end{center}

\begin{rem} \label{r:detail} The Classification
 of locally  compact inductively monothetic  groups (Proposition \ref{p:ind-mon}) yields that the group
 $G/G_0$ is of type (3).
Indeed, it  is a local product of inductively monothetic $p$-Sylow
subgroups of type $\Z(p^n)$, $\Z(p^\infty)$, $\Z_p$, or $\Q_p$.
\end{rem}

\begin{proof} As is described
 in Proposition \ref{p:vector-splitting-int-appr} $G$
has a compact characteristic  subgroup
$N\defi\comp(G_0)=\comp(G)_0$, the unique largest compact connected
subgroup. So by Proposition \ref{p:class(appro-integral)} (QG),
 the quotient $G/N$ is also integrally approximable, and
$\comp(G/N)$ is totally disconnected and therefore is periodic.
As Corollary \ref{c:ind-mon-2} applies to $G/N$, its
factor $\comp(G/N)$ is inductively monothetic. However,
$G/N\cong \R\times \comp(G/N)=\R\times\comp(G)/N$ and
$G_0\cong \R\times N$ we see that $G/G_0\cong \comp(G/N)$.
Thus $G/G_0$ is inductively monothetic.
\end{proof}

It is noteworthy that there is no limitation on the size of
the compact connected abelian group $\comp(G)_0=\comp(G_0)$.
The locally compact abelian group $\comp(G)$ is an extension of
the compact group $\comp(G)_0$ by an inductively monothetic
group.

\bigskip
\section{Sufficient conditions}
\label{s:suff}

In this section we shall prove the following complement to
Theorem \ref{th:necessary} and thereby complete the proof of the main
theorem formulated in the abstract.

\begin{thm} \label{th:sufficient} Let $G=\R\times H$ for a
locally compact abelian group $H$ satisfying the following conditions:
\begin{itemize}
\item[\rm(a)] $H=\comp(H)$ and
\item[\rm(b)] $H/H_0\cong G/G_0$ is inductively monothetic.
\end{itemize}
Then $G$ is integrally approximable.
\end{thm}

\subsection{Various reductions}
\label{ss:reduct}

We shall achieve the proof by reducing the problem step by step.

Firstly, every inductively monothetic group is the directed
union of monothetic subgroups by Proposition \ref{p:ind-mon}.
and so if $H$ satisfies (b) it is of the form
$H=\bigcup_{i\in I}H_i$ with a directed family of subgroups
$H_i\supseteq H_0$ such that $H_i/H_0$ is compact monothetic.
Then, by Proposition \ref{p:class(appro-integral)} (DU),
$\R\times H$ is integrally approximable if all $\R\times H_i$
are integrally approximable for $i\in I$.

Thus from here on,
in place of condition (b),
we shall assume that $H$ satisfies
\begin{itemize}
\item[(c)] $H/H_0$ is monothetic.
\end{itemize}
 After condition (c), $H$ is compact.

Every locally compact abelian group is a strict
projective limit  of Lie groups.
This applies to $H$. Clearly condition (a) holds for all
quotient groups. If $N$ is a compact normal subgroup of $H$
Then $(H/N)_0=H_0N/N$ and thus $(H/N)/(H/N)_0\cong H/H_0N$,
whence $(H/N)/(H/N)_0$ is a quotient group of $H/H_0$
and is therefore inductively monothetic.
Thus if we can show that all for all abelian Lie groups $H$
satisfying (a) and (c), the groups $\R\times H$
 are integrally approximable, then
 the Proposition \ref{p:class(appro-integral)}(PL)
will show that $\R\times H$ is integrally approximable.
Therefore we need to prove Theorem \ref{th:sufficient} for
a compact Lie group $H$ satisfying (a) and (c).

\begin{lem} \label{l:lie-case} Let $H$ be a compact abelian Lie group
satisfying $\mathrm{(a)}$ and $\mathrm{(c)}$.
Then there is are nonnegative integers $m$ and $n$
$$ H \cong \T^m\times \Z(n).$$
In particular, $H$ is monothetic.
\end{lem}

\begin{proof} $H$ is a compact abelian Lie group such that $H/H_0$
is cyclic. Then $H$ has the form asserted (see e.g.\ \cite{hofmorr},
Proposition 2,42 on p.~48 or Corollary 7.58 (iii) on p. 356).

We have seen in the proof of Lemma \ref{l:monothetic} that a compact group
is monothetic if and only if its character group can be injected
into the discrete circle group $\T_d=\R_d\oplus\Q/\Z$. Since
$\hat H \cong \Z^m\oplus \Z(n)$, this condition is satisfied.
\end{proof}

Thus for a proof of Theorem \ref{th:sufficient} it will suffice
to prove

\begin{lem} \label{l:R-plus-mon} If $H$ is monothetic, then
$\R\times H$ is integrally approximable.
\end{lem}

We shall use the Bohr compactification
of the group $\Z$ of integers. This group is also called the {\it
universal monothetic group}.

Here we shall identify  $\hat{\bZ}$ with $\T_d$ (for $\T=\R/\Z$)
and consider the elements $\chi$ of $\bZ$ as characters of $\T_d$.
There is one distinguished character, namely the identity morphism
$\id_\T\colon\T_d\to\T$, and we map $\Z$ naturally and bijectively
onto a dense subgroup of
$\bZ$ via the map
\begin{equation}\label{defn:rho}
\rho\colon\Z\to\bZ,\quad
m\mapsto m\.\id.
\end{equation}

The following Lemma shows that for a proof of Lemma \ref{l:R-plus-mon}
it suffices to prove it for $H=\bZ$.

\begin{lem} \label{l:reduction}  The group $G=\R\times H$ is integrally approximable for any
 monothetic compact subgroup $H$ if and only if it is so for $H=\bZ$.
\end{lem}
\begin{proof} Obviously the latter condition is a necessary one for the
former, so we have to show that it is sufficient. Let $H$ be a
monothetic group. Then there is a morphism
$f\colon \Z\to H$ with dense image.
We have the  canonical dense morphism
$\rho\colon \Z\to\bZ$. By
the universal property of the Bohr compactification there is a unique
morphism $\pi\colon \bZ\to H$ such that $f=\pi\circ\rho$.
Since $f$ has a dense image, this holds for $\pi$, whence by the compactness
of $\bZ$, the morphism is a surjective morphism between compact groups and therefore
is a quotient morphism with a compact kernel.
Therefore there is a quotient morphism with compact
kernel
$\R\times \bZ\to \R\times H$. Hence
by Proposition \ref{p:class(appro-integral)} (QG), $\R\times H$ is
integrally approximable if $\R\times\bZ$ is integrally
approximable.
\end{proof}

Now the following lemma will complete the proof of
Theorem \ref{th:sufficient} and thereby conclude the section with
a proof of the main result of the article.

\begin{lem}[First Key Lemma] \label{l:R-plus-b}
The group $\R\times \bZ$
is approximable by a sequence of integral subgroups.
\end{lem}

\msk

\subsection{Proving the First Key Lemma}
\label{ss:keylem}

The proof of the First Key Lemma requires some technical preparations
in which we use the duality of locally compact abelian groups.
In the process we need to consider the charachter group of
$\R\times\mathrm \bZ$. Here is a reminder of the determination of
the character group of a product:

\begin{lem}\label{dual-group-cartesian-product}
Let $A$ and $B$ be locally compact abelian groups.
Then there is an isomorphism
$\phi\colon\hat A\times\hat B\to(A\times B)\hat{\phantom w}$
such that

\centerline{$\phi(\chi_A,\chi_B)(a,b)=\chi_A(a)-\chi_B(b).$}
\end{lem}

We apply this with $A=\R$ and $B=\bZ$.
For the simplicity of notation we shall denote
the coset $r+\Z\in\T$ of $r\in\R$  by $\overline r$.
We consider $\R$ also as the character group of $\R$ by
letting $r(s)=\overline{rs}\in\T$.
We also identify $\hat\T$ with $\Z$ by considering
$k\in \Z$ as the character defined by $k(\overline r)=\overline{kr}$.
Recall that we consider $\bZ$ as the character group of $\T_d$.
 In the spirit of Lemma \ref{dual-group-cartesian-product}, we
identify $\R\times \T_d$ with the character group $\hat G$ of
$G=\R\times \hat{\T_d}=\R\times\bZ$
by letting $(r,\overline s)$ denote the character of $G$ defined by
 $(r,\overline s)(x,\chi)=\overline{rx}-\chi(\overline s)\in\R/\Z$.
We recall that the identity function $\id_\T\colon\T_d\to\T$ is
a particular character of $\T_d$ and thus is an element of $\bZ$;
indeed  $\id_\T$ is the distinguished generator of $\bZ$.
Now we define
\begin{equation}
\label{eq:Zn}
Z_n=\Z\.(\frac 1 n,\id_\T)\in \cf(G), \quad G=\R\times\bZ.
\end{equation}
Accordingly, $(r,\overline s)\in \R\times\T_d$ belongs to the annihilator
$Z_n^\perp=(\frac 1 n,\id_\T)^\perp$
iff $\overline{\frac r n}-\overline s=0$, that is, iff
$\overline{\frac r n}=\overline s$. This means $\frac r n +\Z=s+\Z$,
and so

\begin{equation}
\label{eq:annih}
(r,\overline s)\in\R\times \T_d\mbox{\quad  is in $Z_n^\perp$ iff\quad }
\frac r n-s\in \Z.
\end{equation}

In an effort to show that $\lim_n Z_n^\perp=\{0\}$ in $\cg(\hat G)$
we consider the following

\begin{lem}[Convergence to the trivial subgroup]\label{convergence-trivial-subgp}
 Let $\Gamma$ be a
 locally compact group with identity $e$, and let $(H_n)_{n\in \N}$ be
a sequence of closed subgroups.  Then  the  following  statements are  equivalent:
\begin{enumerate}
  \item[$(a)$] For each  subnet $(H_{n_j})_{j\in J}$ of $(H_n)_{n\in \N}$
and each convergent
net $(h_{n_j})_{j\in J}$ with $h_{n_j}\in H_{n_j}$ and limit $h$ we have $h=e$.
  \item[$(b)$] $\lim_{n\in \N} H_n=\{e\}$ in $\cg(G)$.
\end{enumerate}
\end{lem}

\begin{proof}
$(a)\Rightarrow (b)$:\quad  We argue by contradiction.
Suppose that there is   a compact subset $K$ of $G$ and
 an open  neighborhood $U$ of  the  identity such that
for each $\alpha\in \N$ there
is an $n_\alpha\in \N$ with $\alpha \leq n_\alpha$ such that
$H_{n_\alpha}\not\in \mathcal{U}(\{e\}; K, U)$. That is,
$H_{n_\alpha}\cap  K \not\subseteq U$ by equation (\ref{eq:E_base}),  and  so
there is an $h_{n_\alpha}\in H_{n_\alpha}$ such that
$h_{n_\alpha}\in K\setminus U$. As  $K\setminus U$
is a compact subset of $G$,  the  net $(h_{n_\alpha})$
admits a subnet  converging to a point $a$ of  $K\setminus U$. Since
$e\not\in K\setminus U$, $a\not= e$, which is a contradiction.

$(b)\Rightarrow (a)$:   Let $(h_{n_j})_{j\in J}$  be  a net
converging to  $h$  and assume  $h_{n_j}\in H_{n_j}$ for  every $j\in J$.
Suppose    $h\ne e$. Now let $K$ be a compact neighborhood of $h$ not containing $e$.
We may assume that $h_{n_j}\in K$ for all $j\in J$.
As $\lim_{n\in \N} H_n=\{e\}$,  there  exists $N\in \N$  such  that  for  each
$n\ge N$ we have $H_n\in\mathcal U(\{e\};K,\Gamma\setminus K)$, that is,
  $H_n\cap K\subseteq\Gamma\setminus K$ by equation (\ref{eq:E_base}),  and this is a contradiction.
\end{proof}

In order to appreciate this lemma, consider the condition

\qquad $(a')$  {\em For each convergent
net $(h_{i})_{i\in I}$ with $h_{i}\in H_{i}$ and limit $h$ we have
$h=e$.}

The  following example will show   that  the  implication
$(a')\Rightarrow (b)$ fails.

\begin{example}
In $\cg(\R)$, let
$$H_n=\left\{
    \begin{array}{ll}
      n\Z, & \hbox{if $n$ is  even;} \\
&\\
      \frac{1}{n}\Z, & \hbox{if $n$ is  odd.}
    \end{array}
  \right.$$
Let  $(h_n)$  be  a  sequence  in  $\R$  converging
to  $h$ and  such  that, for  each  $n$,\ $h_n\in H_n$.
For  any  $n\in \N$, there  is  $k_n\in \Z$  such  that
$h_{2n}= 2n k_n$. As  the  subsequence $(h_{2n})$  converges
to  $h$, $h=0$. So $(a')$ is satisfied.
However,
as the  subsequence $(H_{2n})$  converges  to $\{0\}$
and the  subsequence $(H_{2n+1})$  converges  to $\R$, the
sequence $(H_n)$ is divergent and so $(b)$ fails.
\end{example}

\begin{lem} \label{l:null-convergence}
$\lim_{n\in\N} Z_n^\perp =\{0\}$.
\end{lem}
\begin{proof}
We shall apply Lemma \ref{convergence-trivial-subgp} and assume
that we have a  net $(n_j)_{j\in J}$ cofinal in $\N$ such that
$(r_j,\overline s_j)_{j\in J}$ is a convergent net with

\centerline{$(r_j,\overline{s_j})\in Z_{n_j}^\perp$ for all $j\in J$
and with $(r,\overline s)=\lim_{j\in J} (r_j,\overline{s_j})$.}

From Equation (\ref{eq:annih}) we know that
$\frac{r_j}{n_j}-s_j\in \Z$. Since $r_j\to r$,  $s_j\to s$,
and $n_j\to\infty$, we conclude  $s\in \Z$ and thus $\overline s=0$.
Further $s\in \Z$ and $r_j/n_j\to 0$ imply the existence
of a $j_0$ such that
$j_0\le j$ implies $s_j=s$. Then $r_j/n_j$ is an integer,
and so for large enough $j$
we have $r_j=0$ which implies $r=0$.
Thus $\lim_{j\in J}(r_j,\overline{s_j})=0$
and by Lemma \ref{convergence-trivial-subgp} this shows that
$\lim_{n\in\N}Z_n^\perp =\{0\}$ which we had to show.
\end{proof}

Now we are ready for  proof of the First Key Lemma:
There is a sequence $(Z_n)_{n\in \N}$ in $\cf(G)$,
$Z_n=\Z\.(\frac 1 n,\id_\T)$
for which $(Z_n^\perp)_{n\in \N}$ converges  to $\{0\}$
in $\cg(\hat G)$ by Lemma \ref{l:null-convergence}.
Now we apply Pontryagin-Chabauty Duality in the form
of Proposition \ref{Pon-Cha-duality} and conclude
\begin{equation}
\label{eq:sequential-appro}
\lim_{n\in\N}Z_n =G\mbox{ in }\cg(G)
\mbox{ for }G=\R\times\bZ.
\end{equation}
This concludes the proof of the First Key Lemma \ref{l:R-plus-b}
and thus also finishes the proof of Theorem \ref{th:sufficient}.

We can summarize the main result as follows:

\begin{thmA} A locally compact group $G$ is integrally
approximable if and only if it is either discrete, in which case it is isomorphic
to a nonsingleton subgroup of $\Q$, or else 
it is abelian and of the form $G\cong \R\times \comp(G)$
where $G/G_0\cong \comp(G)/\comp(G)_0$ is periodic
 inductively monothetic.
 \end{thmA}

We recall that the groups $\comp(G)_0$ range through all compact connected abelian groups
and that the periodic inductively monothetic groups were classified in
Proposition \ref{p:ind-mon} (3).

We mention in passing that Theorem A yields a characterisation of the the group
$\R$ in the class of locally compact groups. For this purpose let us call
a topological group {\it compact-free} if it does not contain a nonsingleton
compact subgroup.

\begin{cor} \label{c:compactfree} For a locally compact group $G$ the
following conditions are equivalent:
\begin{enumerate}
\item $G\cong\R$ or else is isomorphic to a nonsingleton subgroup of the 
      discrete group $\Q$.
\item $G$ is compact-free and integrally approximable.
\end{enumerate} 
\end{cor}

A   similar characterization using the property of being
compact-free  was suggested by Chu in \cite{chu}.

\section{Complement: Groups approximable by real  subgroups}

We classified locally compact groups $G$ for which every neighborhood of
$G\in \cg(G)$ contains a subgroup $H$ of $G$ isomorphic to $\Z$. We called
such groups {\it integrally approximable}. Now we
shall do the same for groups $G$ with the same property except that
$\Z$ is replaced by $\R$. We need a name for these groups that are
approximated by subgroups of (real) {\it numbers}. For this
purpose let us denote by $\cR(G)$ the subspace of $\cg(G)$ containing all
 subgroups isomorphic to the  group $\R$ of all real numbers, the
{\it real vector group of dimension} $1$.

\begin{defn} \label{numeral-subgroups}
A locally compact group $G$ is said to be
\textit{numerally approximable}
if $G\in\overline{\cR(G)}$.
\end{defn}

Trivially, $\R$ is numerally approximable.
The theory and classification of numerally approximable groups is
in most ways simpler than that of integrally approximable groups.
The basic aspects are completely analogous to the former and
that allows us now to proceed more expeditiously.

\begin{example} \label{p:vector-groups-again}
For $n\in\N$ and $G=\R^n$, the following statements are equivalent:
\begin{itemize}
\item[(1)] $G$ is numerally approximable.
\item[(2)] $n=1$.
\end{itemize}
\end{example}
\begin{proof} Trivially, (2) implies (1). Conversely,
if (1) is satisfied, an inspection of the proof of (1)\implies(2)
for Example \ref{p:vector-groups} shows that it applies to
the next to the last sentence, where $\Z$ needs to be replaced
by $\R$ to applie literally.
\end{proof}

\begin{lem}\label{encadrement}
Let  $G$  be  a locally  compact group  and
$(A_i)_{i\in I}$, $(B_i)_{i\in I}$  two  nets
converging to  $A$  and  $B$  respectively.
If $A_i\subseteq B_i$ holds eventually, then  $A\subseteq B$.
\end{lem}

\begin{proof}  By way of contradiction, suppose  that $A$  is  not  a
subgroup  of  $B$.  Let  $x\in A\setminus  B$  and  $U$ a
relatively  compact  open  neighborhood of  $e$  such  that
$\overline U x\cap B =\emptyset$. As $B_i \to B$, there
exists  $i_0\in I$  such  that  for  any $i\ge  i_0$  we
have $\overline{U} x\cap
B_i=\emptyset$  and  so $Ux \cap  A_i =\emptyset$,
which is  a  contradiction because
the set if all closed subgroups meeting the open set $Ux$ is
an open neighborhood of $A$.
\end{proof}

In particular, this lemma implies that the limit of
a net of connected closed subgroups of a locally compact group $G$
is contained in the identity component $G_0$ of $G$, whence
a group $G$ which is approximable by real subgroups is necessarily
connected.  Using also Proposition \ref{space-abelian-closed},
we conclude:

\begin{lem}\label{l:realappro:connected}
Every  numerally approximable locally  compact  group
 is abelian and connected.
\end{lem}

In view of Proposition \ref{p:vector-splitting} we may rephrase this
as follows:

\begin{thm} \label{th:necessary-numeral} Let $G$ be a numerally
approximable locally compact group. Then $\comp(G)$ is
compact and connected and
$$G\cong \R\times \comp(G).$$
\end{thm}

\vskip-20pt

\begin{gather}
\begin{aligned}
\xymatrix{
&G \ar@{-}[dr] \ar@{-}[dl]&&\\
\R\ar@{-}[dr]  && C\ar@{-}[dl] \\
&0
}\end{aligned}
\label{Fig3}
\end{gather}
 \begin{center}
 {$C=\comp(G),\quad G=\R\oplus C.$}
\end{center}

\bsk

The second part of the structure theorem for numerally approximable
groups, saying that every group $\R\times C$ with any compact connected abelian group $C$
is numerally approximable is proved in a reduction procedure similar to the one
we used for integrally approximable groups.

\begin{thm} \label{th:sufficient-numeral} Let $G=\R\times C$ for
a compact connected abelian group $C$. Then $G$ is numerally approximable.
\end{thm}
\begin{proof}
Every compact connected group is a directed union of  compact connected
monothetic groups (see e.g. \cite{hofmann-tran-amer}, Theorem I or \cite{hofmorr},
Theorem 9.36(ix), pp.~479f.). Thus $G$ is the directed union  of
subgroups $\R\times M$ where $M$ is compact connected monothetic.
The closure lemma Proposition \ref{p:class(appro-integral)} (DU)
is easily seen to apply to numerally approximable groups in place of
integrally approximable groups. Therefore
it is no loss of generality to assume that $C$ is compact connected
monothetic. Then it is a quotient of $\bR$, the Bohr compactification
of $\R$. (Indeed $\hat C$ is a discrete torsion free group of
rank $\le 2^{\aleph_0}$ and thus is a subgroup of $\R_d$ (the discrete
reals), and thus $C$ is a quotient of $\hat{\R_d}\cong \bR$.)
Since the closure lemma Proposition \ref{p:class(appro-integral)} (QG)
again applies to numerally approximable groups in place of integrally
approximable groups, the proof will be
complete if it is shown that $\R\times\bR$ is numerally approximable.
This will be done in the Second Key Lemma that follows.
\end{proof}

The proof of the Theorem is therefore reduced to showing that
one special group in numerally approximated:

\begin{lem}[Second Key Lemma] \label{l:key-lemma-2}
The group $G=\R\times\bR$ is numerally
approximable.
\end{lem}

Before we prove this Second  Key Lemma, we review the duality aspects of
the present situation.

 From Lemma \ref{dual-group-cartesian-product} we recall
that the dual $\hat G$ of $G$ may be identified with $\R\times\R_d$
where, as before we identify $\hat \R$ and $\R$.
 Let $f\colon\R\to\bR$ the canonical one-parameter
subgroup of $\bR$, namely the dual of $\id_\R\colon\R_d\to\R$.
For each natural number $n\in\N$, the morphism
$n\.f$ defined by $(n\.f)(r)=n\.f(r)$ (in the additively
written) abelian group $\bR$) is the dual of the
morphism $n\.\id_\R\colon \R_d\to\R$ which is
just multiplication by $n$.
 We let  the subgroup $R_n\le \R\times\bR$
be the graph of $n\.f$, that is,
$$R_n= \set{(r,n\.f(r))}{r\in\R}.$$
Clearly, $R_n$ is isomorphic  to $\R$, since
the projection of a graph of a morphism onto its domain
is always an isomorphism.

 We claim that $G=\lim_n R_n$ in $\cg(G)$; this claim
will finish the proof of the Second Key Lemma.
 We shall prove the claim
by showing that $\lim_n R_n^\perp=\{(0,0)\}$ in
$\cg(\hat G)$.
This will prove the claim by Pontryagin-Chabauty-Duality
in the form of Proposition \ref{Pon-Cha-duality}.

We need information on the annihilator of a graph:

\begin{lem}\label{GraphAdjoint} Let $A$ and $B$ be locally compact groups and
$f\colon A\to B$ a morphism. Define $\Gamma=\set{(a,f(a))}{a\in A}
\subseteq A\times B$ to denote the graph of $f$. We identify
$(A\times B)\hat{\phantom w}$
with $\hat A\times \hat B$ (via $\rho$ as in Lemma \ref{dual-group-cartesian-product}).
Then the graph $\set{(\hat f(b),b)}{b\in B}\subseteq A\times B$ of the
adjoint morphism $\hat f\colon \hat B\to \hat A$
is the annihilator $\Gamma^\perp$ of $\Gamma$.
\end{lem}
\begin{proof} An element $(\chi_A,\chi_B)\in (A\times B)^\times$ is in
$\Gamma^\perp$ if  and  only  if, for any $a\in A$,
$$0= (\chi_A,\chi_B)(a,f(a))
=\chi_A(a)-(\chi_B\circ f)(a)=(\chi_A-\hat f(\chi_B))(a),$$
that is $\chi_A=\hat f(\chi_B)$.
\end{proof}

Applying this lemma to the graph $R_n$ we see that
\begin{equation} \label{annihilation}
 R_n^\perp=\{(nr,r):r\in\R\}=
\{(r,\frac r n):r\in\R\}
\subseteq \R\times\R_d.\end{equation}
Note that $R_n^\perp=\graph(r\mapsto \frac r n)$.

\begin{lem} \label{l:null-convergence-2} We  have $\lim_n R_n^\perp=\{(0,0)\}$ in $\cg(\hat G)$.
\end{lem}
\begin{proof} As in the proof of Lemma \ref{l:null-convergence} we invoke
Lemma \ref{convergence-trivial-subgp} and consider
a  net $(n_j)_{j\in J}$ cofinal in $\N$ such that
$(r_j,s_j)_{j\in J}$ is a convergent net in $\hat G=\R\times\R_d$
such that according to Equation (\ref{annihilation}) we have
\begin{itemize}
\item[(a)]  $(r_j,s_j)\in R_{n_j}^\perp=\{(r,\frac r n_j): r\in\R\}$
 for all $j\in J$, and
\item[(b)] $(r,s)=\lim_{j\in J} (r_j,s_j)$ in $\hat G=\R\times\R_d$.
\end{itemize}
We must show that this implies $r=s=0$; then
Lemma \ref{convergence-trivial-subgp} completes the proof of
the lemma.

Now (a) implies $s_j=r_j/n_j$ for all $j\in J$, and (b) yields,
firstly, that $r=\lim_j r_j$ in $\R$ and, secondly, that in view
of the discreteness of $\R_d$ the net of the $r_j/n_j=s_j$ in $\R_d$ are
eventually constant, say $=t$ for $j>j_0$ for some $j_0$.
For these $j$ we now have $r_j=tn_j$, and so $r=\lim_j tn_j$.
Since the $n_j$ increase beyond all bounds, this implies $r=t=0$,
and so $r_j=0$ for $j>j_0$. Accordingly, $s_j=0$ for $j>j_0$,
and thus $s=\lim_j s_j=0$ as well. This completes the proof.
\end{proof}

By our earlier remarks, this shows $\lim_n R_n=G$ and thus completes
the proof of the Second Key Lemma \ref{l:key-lemma-2} and thereby also
completes the proof of Theorem \ref{th:sufficient-numeral}.

We can summarize the material on numerally approximable groups as follows:

\msk

\begin{thmB} A locally compact group $G$ is
numerally approximable if and only if it is of the form $G\cong\R\times C$
with a compact connected abelian group $C$.
\end{thmB}

\msk

By Pontryagin Duality, the groups  $C$ range through a class equivalent to
the class of all torsion free abelian (discrete) groups.

\msk

A comparison of Main Theorems A and B allow us to draw the following
conclusion:

\begin{cor} \label{c:final} If a locally compact group is
numerally approximable, then it is integrally approximable, while the reverse
is not generally true.
\end{cor}

\begin{rem} Locally compact groups of the form $\R\times C$ for a
compact connected group $C$ have been called {\it two-ended}
by Freudenthal (\cite{freud}).
\end{rem}

\section{An alternate proof  of Corollary \ref{c:final}}

Corollary 6.10 was proved above in a roundabout fashion. We therefore present
a different direct way to arrive at the same conclusion.

\subsection{Iterated Limit Theorem}

The Iterated Limit Theorem for nets deals with  the following
data:

Let $I$ be a directed set and $(J_i)_{i\in I}$ a family of
directed sets indexed by $I$. Assume that for each $i\in I$
we are given a converging net $(x_{ij})_{j\in J_i}$ in a topological
space $X$.
The limits $r_i=\lim_{j\in J_i}x_{ij}$, $i\in I$ form a net $(r_i)_{i\in I}$
which may or may not converge. If it does, we have what is sometimes called
an iterated limit
$$ r=\lim_{i\in I}\lim_{j\in J_i} x_{ij}.$$

There is no harm in our assuming that the sets $J_i$ are pairwise disjoint.
(If necessary we can replace each $J_i$ by $\{i\}\times J_i$!). Now
$D\defi\dot\bigcup_{k\in I}J_k$ is a disjoint union of the fibers $J_i$
of the fibration $\pi\colon D\to I$, $\pi(d)=i$ iff $d\in J_i$.
A function $\sigma\colon I\to D$ is a {\it  section} if
$\pi\circ \sigma=\id_I$, the identity function of $I$. The
set of sections is  $\prod_{i\in I} J_i$.
The set $D$ is partially ordered lexicographically:
$$d\le d'\hbox{ if either }\pi(d)< \pi(i'),
\hbox{ or }\pi(i)=\pi(i')\hbox{ and  }d\le d'
\hbox{ in }J_{\pi(d)}.\leqno(1)$$

These data taken together result in
a net $(x_d)_{d\in D}$ which converges fiberwise,
and the limits
along the fibers form a convergent net.

\ssk

In applications such as ours it is desirable to construct from the given
data a subnet
$(y_p)_{p\in P}$   of the net $(x_d)_{d\in D}$ in $X$
for some directed set $P$ such that
$$ r=\lim_{p\in P} y_p.$$

\begin{lem} Let $P=I\times\prod_{k\in I}J_k$. For
$p=(i,\sigma)\in P$
we define $d_p =\sigma(i)\in D$.
Then
$p\mapsto d_p:P\to D$ is a cofinal function between directed sets,
that is, for each $d\in D$ there is a $p_0\in P$ such that $p_0\le p$ implies
$d\le d_p$.
\end{lem}

\begin{proof}  Let $d\in D$, that is, $d\in J_{\pi(d)}$.
 We have to find a
$p_0=(i,\sigma)\in P$ such that $p_0\le p=(k,\tau)$ in $P$
implies $\sigma(i)\le \tau(k)$ in $D$.

Now for $k\ne i$ use the Axiom of Choice to  select an arbitrary $j_k\in J_k$.
Define $\sigma\in\prod_{i'\in I}J_{i'}$ by
$$\sigma(k)=\begin{cases}\sigma(\pi(d)), & \mbox{if }  k=\pi(i)\\
                j_k. & \mbox{otherwise.}
              \end{cases}
$$
Set $p_0=(\pi(d),\sigma)$. Now if $p_0\le p=(k,\tau)$ with
$\pi(d)\le k$ and $\sigma\le \tau$
we claim that $d=\sigma(\pi(d)))\le d_p=\tau(k))$. Indeed,
 if $\pi(d)<k$ then $d\le \tau(k)=d_p$ by the order on $D$, and if
$i=k$ then $d=\sigma(\pi(d)) \le \tau(\pi(d))$ since $\sigma\le \tau$.
This completes the proof.
\end{proof}

\msk

The subnet is constructed so as to be convergent if the iterated
limit exists and to have the same limit.
This is the so-called
{\it Iterated Limit Theorem} in whose formulation we use the
notation introduced above.
\begin{thm}\label{IteratedLimitTheorem}
Assume that the  net $(x_d)_{d\in D}$ converges fiberwise
and that the fiberwise limits converge as well.
Then it has a subnet $(y_p)_{p\in P}$,
$y_p=x_{d_p}$, $d_{i,\sigma}=\sigma(i)$,
such that
$$\lim_{p\in P}y_p=\lim_{i\in I}\lim_{d\in J_i}x_d.$$
\end{thm}
\begin{proof}
For a proof see \cite{kelley}, p.~69.
\end{proof}
In this remarkable fact about nets and their convergence it
is noteworthy that the index set $P$ is vastly larger
than the already large index set $D$.

\msk

A frequent special case is that the index sets $J_i$ all agree with
one and the same  index set $J$ in which case we have
$D=I\times J$,  $J_i=\{i\}\times J$,  and $P=I\times J^I$.

\msk

Now  we  present an alternate proof  of Corollary \ref{c:final}.

\begin{proof}  Let $G$  be a
numerally aproximable locally compact group  and  let $(R_j)_{j\in J}$  be  a net such  that
$G=\lim_{j\in J} R_j$ and $R_j\cong \R$.  For each $j\in J$ we have $R_j=\lim_{n\in \N} Z_{(j,n)}$ for a
sequence $Z_{(j, n)}\cong \Z$. Therefore
\begin{equation*}
G=\lim_{j\in J}\lim_{n\in\N} Z_{(j,n)}.
\end{equation*}
Then there exists a subnet $(Z_p)_{p\in P}$
of the net $(Z_{(j, n)})_{(j,n)\in J\times \N}$
such that
\begin{equation*}
G=\lim_{p\in P} Z_{p}.
\end{equation*}
by the
Theorem of the Iterated Limit.
\end{proof}

A review of the iterated limit theorem together with an application of it in
the context of our present topic may be of independent interest.

\bibliographystyle{amsplain}

\end{document}